\documentclass[11pt,twoside]{article}
\usepackage{latexsym}
\usepackage{amssymb,amsbsy,amsmath,amsfonts,amssymb,amscd}
\usepackage{epsfig, graphicx,subfigure}

\usepackage{colortbl}
\newcommand {\wt}[1] {{\widetilde #1}}

\newcommand{\commentout}[1]{}
\newcommand{\R}{\mathbb{R}}
\newcommand{\N}{\mathbb{N}}

\newcommand{\1}{{\mathchoice {\rm 1\mskip-4mu l} {\rm 1\mskip-4mu l}
{\rm 1\mskip-4.5mu l} {\rm 1\mskip-5mu l}}}

\newcommand {\al} {\alpha}
\newcommand {\e}  {\varepsilon}

\newcommand {\vp} {\varphi}
\newcommand {\lb} {\lambda}
\newcommand {\Chi} {{\bf \raise 2pt \hbox{$\chi$}} }

\newcommand {\sgn} { {\rm sgn} }

\newcommand {\U} { {\mathcal U} }
\newcommand {\f}   {\frac}
\newcommand {\p}   {\partial}

\newcommand{\dis}{\displaystyle}

\newcommand{\beq}{\begin{equation}}
\newcommand{\beqa} {\begin{array}{rl}}
\newcommand{\eeq}{\end{equation}}
\newcommand{\eeqa}{\end{array}}
\newtheorem{theorem}{Theorem}
\newtheorem{lemma}{Lemma}

\newtheorem{remark}{Remark}

\newcommand{\qed}{{ \hfill
                       {\unskip\kern 6pt\penalty 500
                       \raise -2pt\hbox{\vrule\vbox to 6pt{\hrule width 6pt
                       \vfill\hrule}\vrule} \par}   }}
\title{\LARGE Eigenelements of a General Aggregation-Fragmentation Model.}

\author{Marie Doumic \thanks{INRIA Rocquencourt, projet BANG, Domaine de Voluceau, BP 105, F-78153 Rocquencourt, France.} \hspace{1cm} Pierre Gabriel \thanks{Universit\'e Pierre et Marie Curie-Paris 6, UMR 7598 LJLL, BC187, 4, place de Jussieu, F-75252 Paris cedex 5; corresponding author, email: gabriel@ann.jussieu.fr}}

\date{\today}

\begin{document}
\maketitle
\pagestyle{plain}
\pagenumbering{arabic}

\begin{abstract}
We consider a linear integro-differential equation which arises to describe both aggregation-fragmentation processes and cell division. We prove the existence of a solution $(\lb,\U,\phi)$ to the related eigenproblem. Such eigenelements are useful to study the long time asymptotic behaviour of solutions as well as the steady states when the equation is coupled with an ODE. Our study concerns a non-constant transport term that can vanish at $x=0,$ since it seems to be relevant to describe some biological processes like proteins aggregation. Non lower-bounded transport terms bring difficulties to find $a\ priori$ estimates. All the work of this paper is to solve this problem using weighted-norms.
\end{abstract}

\


\

\noindent{\bf AMS Class. No.} 35A05, 35B40, 45C05, 45K05, 82D60, 92D25

\

\section{Introduction}

Competition between growth and fragmentation is a common phenomenon for a structured population. It arises for instance in a context of cell division (see, among many others, \cite{adimy, basse, Bekkal1, Bekkal2, Doumic, farkas, GW2, MetzDiekmann, PTou}), polymerization (see \cite{biben, destaing}), telecommunication (see \cite{Baccelli}) or neurosciences (see \cite{PPS}). It is also a mechanism which rules the proliferation of prion's proteins (see \cite{CL1, Greer, LW}). These proteins are responsible of spongiform encephalopaties and appear in the form of aggregates in infected cells. Such polymers grow attaching non infectious monomers and converting them into infectious ones. On the other hand they increase their number by splitting. 

To describe such phenomena, we write the following integro-differential equation,
\beq\label{eq:temporel}\left \{ \begin{array}{l}
\dis\dfrac{\partial}{\partial t} u(x,t) + \dfrac{\partial}{\partial x} \big(\tau(x) u(x,t)\big) + \beta(x) u(x,t) = 2 \int_{x}^{\infty} \beta(y) \kappa (x,y) \, u(y,t) \, dy, \qquad x \geqslant0,
\\
\\
u(x,0)=u_0(x),
\\
\\
u(0,t)=0.
\end{array}\right.\eeq
The function $u(x,t)$ represents the quantity of individuals (cells or polymers) of structured variable (size, protein content...) $x$ at time $t.$ These individuals grow (\emph{i.e.}, polymers aggregate monomers, or cells increase by nutrient uptake for instance) with the rate $\tau(x).$ Equation \eqref{eq:temporel} also takes into account the fragmentation of a polymer (or the division of a cell) of size $y$ into two smaller polymers of size $x$ and $y-x.$ This fragmentation occurs with a rate $\beta(y)$ and produce an aggregate of size $x$ with the rate $\kappa(x,y).$ Equation \eqref{eq:temporel} is a particular case of the more general one
\beq\label{eq:general}\dfrac{\p}{\p t} u(x,t) + \dfrac{\p}{\p x} \big(\tau(x) u(x,t)\big) + [\beta(x)+\mu(x)] u(x,t) = n \int_{x}^{\infty} \beta(y) \kappa (x,y) \, u(y,t) \, dy, \qquad x \geqslant x_0,\eeq
with the bound condition $u(x_0,t)=0$ (see \cite{Banasiak, CL1, LW}). Here, polymers are broken in an average of $n>1$ smaller ones by the fragmentation process, there is a death term $\mu(x)\geq0$ representing degradation, and a minimal size of polymers $x_0$ which can be positive. This more general model is biologically and mathematically relevant in the case of prion proliferation and is used in \cite{CL2, CL1, Greer, LW} with a coupling to an ODE. Our results remain true for this generalization.

A fundamental tool to study the asymptotic behaviour of the population when $t\to\infty$ is the existence of eigenelements ($\lb,\U,\phi$) solution of the equation

\beq
\label{eq:eigenproblem}
\left \{ \begin{array}{l}
\displaystyle \f{\p}{\p x} (\tau(x) \U(x)) + ( \beta (x) + \lb) \U(x) = 2 \int_x^\infty \beta(y)\kappa(x,y) \U(y) dy, \qquad x \geqslant0,
\\
\\
\tau\U(x=0)=0 ,\qquad \U(x)\geq0, \qquad \int_0^\infty \U(x)dx =1,
\\
\\
\displaystyle -\tau(x) \f{\p}{\p x} (\phi(x)) + ( \beta (x) + \lb) \phi(x) = 2 \beta(x) \int_0^x \kappa(y,x) \phi(y) dy, \qquad x \geqslant0,
\\
\\
\phi(x)\geq0, \qquad \int_0^\infty \phi(x)\U(x)dx =1.
\end{array} \right.
\eeq

For the first equation (equation on $\U$) we are looking for ${\mathcal D}'$ solutions defined as follows : $\U\in L^1(\R^+)$ is a ${\mathcal D}'$ solution if $\forall\varphi\in{\mathcal C}^\infty_c(\R^+),$
\beq\label{eq:D'eigenproblem}
\dis -\int_0^\infty\tau(x)\U(x)\p_x\varphi(x)\,dx + \lb\int_0^\infty\U(x)\varphi(x)\,dx =  \int_0^\infty\beta(x)\U(x)\Bigl(2\int_0^\infty\varphi(y)\kappa(y,x)\,dy-\varphi(x)\Bigr)\,dx.
\eeq
Concerning the dual equation, we are looking for a solution $\phi\in W^{1,\infty}_{loc}(0,\infty)$ such that the equality holds in $L^1_{loc}(0,\infty),\ i.e.$ almost everywhere.

When such elements exist, the asymptotic growth rate for a solution to \eqref{eq:temporel} is given by the first eigenvalue $\lb$ and the asymptotic shape is given by the corresponding eigenfunction $\U.$ More precisely, it is proved for a constant fragmentation rate $\beta$ that $u(x,t)e^{-\lb t}$ converges exponentially fast to $\rho\U(x)$ where $\rho=\int u_0(y)dy$ (see \cite{LP,PR}). For more general fragmentation rates, one can use the dual eigenfunction $\phi$ and the so-called ''General Relative Entropy`` method introduced in \cite{MMP1,BP}. It provides similar results but without the exponential convergence, namely that $$\int_0^\infty \bigl|u(y,t)e^{-\lambda t}-\langle u_0,\phi\rangle{\mathcal U}(y)\bigr|\phi(y)\,dy\underset{t\to\infty}{\longrightarrow}0$$ where $\langle u_0,\phi\rangle=\int u_0(y)\phi(y)dy$ (see \cite{MMP1,MMP2}).

The eigenvalue problem can also be used in nonlinear cases, such as prion proliferation equations, where there is a quadratic coupling of Equation \eqref{eq:temporel} or \eqref{eq:general} with a differential equation.  In \cite{CL2, CL1, Engler, Pruss} for instance, the stability of steady states is investigated. The use of entropy methods in the case of nonlinear problems remains however a challenging and widely open field (see \cite{PTum} for a recent review).

\

Existence and uniqueness of eigenelements has already been proved for general fragmentation kernels $\kappa(x,y)$ and fragmentation rates $\beta(x),$ but with very particular polymerization rates $\tau(x),$ namely constant ($\tau\equiv1$ in \cite{BP}), homogeneous ($\tau(x)=x^\mu$ in \cite{M1}) or with a compact support ($Supp\,\tau=[0,x_M]$ in \cite{Doumic}).

The aim of this article is to consider more general $\tau$ as \cite{CL1, Silveira} suggest. Indeed, there is no biological justification to consider specific shapes of $\tau$ in the case when $x$ represents a size (mass or volume) or some structuring variable and not the age of a cell (even in this last case it is not so clear that $\f{dx}{dt}=1,$ since biological clocks may exhibit time distorsions). For instance, for the prion proteins, the fact that the small aggregates are little infectious (see \cite{Lenuzza,Silveira}) leads us to include the case of rates vanishing at $x=0.$

Considering fully general growth rates is thus indispensable to take into account biological or physical phenomena in their full diversity. 
The proof of \cite{BP} can be adapted for non constant rates but still positive and bounded ($0<m<\tau(x)<M$). The paper \cite{M1} gives results for $\tau(0)=0,$ but for a very restricted class of shape for $\tau.$ The paper \cite{Doumic} gives results for $\tau$ with general shape in the case where there is also an age variable (integration in age then allows to recover Problem \eqref{eq:temporel}), but requires a compact support and regular parameters. Here we consider polymerization rates that can vanish at $x=0,$ with general shape and few regularity for the all parameters ($\tau,\ \beta$ and $\kappa$).

From a mathematical viewpoint, relaxing as far as possible the assumptions on the rates $\tau,\kappa,\beta,$ as we have done in this article, also leads to a better understanding of the intrinsic mechanisms driving the competition between growth and fragmentation.

\begin{theorem}[Existence and Uniqueness]\label{th:eigenelements}
Under assumptions \eqref{as:kappa1}-\eqref{as:betatauinf},
there exists a unique solution $(\lb,\U,\phi)$ (in the sense we have defined before) to the eigenproblem \eqref{eq:eigenproblem} and we have
$$\lb>0,$$
$$x^\al\tau\U\in L^p(\R^+),\quad\forall \al\geq-\gamma,\quad\forall p\in[1,\infty],$$
$$x^\al\tau\U\in W^{1,1}(\R^+),\quad\forall \al\geq0$$
$$\exists k>0\ s.t.\ \f{\phi}{1+x^k}\in L^\infty(\R^+),$$
$$\tau\f{\p}{\p x}\phi\in L_{loc}^\infty(\R^+).$$
\end{theorem}

The end of this paper is devoted to define precisely the assumptions and prove this theorem. It is organized as follows : in Section \ref{se:coefficients} we describe the assumptions and give some examples of interesting parameters. In Section \ref{se:proof} we prove Theorem \ref{th:eigenelements} using $a\ priori$ bounds on weighted norms and then we give some consequences and perspectives in Section \ref{se:csq}. The proof of technical lemmas and theorem can be found in the Appendix.

\

\section{Coefficients}\label{se:coefficients}

\subsection{Assumptions}

For all $y\geq0,\ \kappa(.,y)$ is a nonnegative measure with a support included in $[0,y].$ We define $\kappa$ on $(\R_+)^2$ as follows : $\kappa(x,y)=0\ \text{for}\ x>y.$ We assume that
 for all continuous function $\psi,$ the application $f_\psi:y\mapsto\int\psi(x)\kappa(x,y)\,dx$ is Lebesgue measurable.\\
The natural assumptions on $\kappa$ (see \cite{Greer} for the motivations) are that polymers can split only in two pieces which is taken into account by
\beq\label{as:kappa1}\int\kappa(x,y) dx = 1.\eeq
So $\kappa(y,.)$ is a probability measure and $f_\psi\in L^\infty_{loc}(\R^+).$ The conservation of mass imposes
\beq\label{as:kappa2}\int x\kappa(x,y) dx = \frac y 2,\eeq
a property that is automatically satisfied for a symetric fragmentation ($i.e.\ \kappa(x,y)=\kappa(y-x,y)$) thanks to \eqref{as:kappa1}. For the more general model \eqref{eq:general}, assumption \eqref{as:kappa2} becomes $\int x\kappa(x,y) dx = \frac y n$ to preserv the mass conservation.\\
We also assume that the second moment of $\kappa$ is less than the first one
\beq\label{as:kappa3}\int\f{x^2}{y^2} \, \kappa(x,y) dx \leq c < 1/2\eeq
(it becomes $c<1/n$ for model \eqref{eq:general}). We refer to the Examples for an explanation of the physical meaning.

\

For the polymerization and fragmentation rates $\tau$ and $\beta,$ we introduce the set
$${\mathcal P}:=\bigl\{f\geq0\,:\,\exists\mu,\nu\geq0,\ \limsup_{x\to\infty}x^{-\mu}f(x)<\infty\ \text{and}\ \liminf_{x\to\infty}x^\nu f(x)>0\bigr\}$$
and the space
$$L^1_0:=\bigr\{f,\ \exists a>0,\ f\in L^1(0,a)\bigl\}.$$
We consider
\beq\label{as:betatauspace}\beta\in L^1_{loc}(\R^{+*})\cap{\mathcal P},\qquad
\exists\al_0\geq0\ s.t.\ \tau\in L^\infty_{loc}(\R^{+},x^{\al_0}dx)\cap{\mathcal P}\eeq
satisfying
\beq\label{as:taupositivity}\forall K\ \text{compact of}\ (0,\infty),\ \exists m_K>0\quad s.t.\quad\tau(x)\geq m_K\ \text{for}\ a.e.\ x\in K\eeq
(if $\tau$ is continuous, this assumption \eqref{as:taupositivity} is nothing but saying that for all $x>0,\ \tau(x)>0$) and
\beq\label{as:betasupport}\exists b\geq0,\quad Supp\beta=[b,\infty).\eeq
Assumption \eqref{as:betasupport} is necessary to prove uniqueness and existence for the adjoint problem.

\

To avoid shattering (zero-size polymers formation, see \cite{Banasiak,LW}), we assume
\beq\label{as:kappatau}\exists\, C>0,\gamma\geq0\quad s.t.\qquad\int_0^x\kappa(z,y)\,dz\leq \min\Bigl(1,C\Bigl(\f x y\Bigr)^\gamma\Bigr)\qquad\text{and}\qquad\f{x^\gamma}{\tau(x)}\in L^1_0\eeq
which links implicitely $\tau$ to $\kappa,$ and also
\beq\label{as:betatau0}\f{\beta}{\tau}\in L^1_0.\eeq
On the other hand, to avoid forming infinitely long polymers (gelation phenomenon, see \cite{EscoMischler1,EscoMischler2}), we assume
\beq\label{as:betatauinf}\lim_{x\rightarrow +\infty}\f{x\beta(x)}{\tau(x)}=+\infty.\eeq

\begin{remark}\label{rk:kappa}
In case when \eqref{as:kappatau} is satisfied for $\gamma>0,$ then \eqref{as:kappa3} is automatically fulfilled (see Lemma \ref{lm:kappa} in the Appendix).
\end{remark}

\

\subsection{Examples}

First we give some examples of coefficients which satisfy or not our previous assumptions.\\
For the fragmentation kernel, we first check the assumptions \eqref{as:kappa1} and \eqref{as:kappa2}. They are satisfied for autosimilar measures, namely $\kappa(x,y)=\f1y\kappa_0(\f xy),$ with $\kappa_0$ a probability measure on $[0,1],$ symmetric in $1/2.$ Now we exhibit some $\kappa_0.$

\

\noindent{\bf General mitosis :} a cell of size $x$ divides in a cell of size $rx$ and one of size $(1-r)x$ (see \cite{M2})
\beq\label{ex:kappar}\kappa_0^r=\f12(\delta_{r}+\delta_{1-r})\qquad\text{for}\qquad r\in[0,1/2].\eeq
Assumption \eqref{as:kappatau} is satisfied for any $\gamma>0$ in the cases when $r\in(0,1/2].$ So \eqref{as:kappa3} is also fulfilled thanks to Remark \ref{rk:kappa}. The particular value $r=1/2$ leads to equal mitosis ($\kappa(x,y)=\delta_{x=\f y2}$).\\
The case $r=0$ corresponds to the renewal equation ($\kappa(x,y)=\f12(\delta_{x=0}+\delta_{x=y})$). In this case, we cannot strictly speak of mitosis because the size of the daughters are $0$ and $x.$ It appears when $x$ is the age of a cell and not the size. This particular case is precisely the one that we want to avoid with assumption \eqref{as:kappa3} ; it can also be studied seperately with different tools (see \cite{PTum} for instance). For such a fragmentation kernel, assumption \eqref{as:kappatau} is satified only for $\gamma=0,$ and the moments $\int z^k\kappa_0(z)dz$ are equal to $1/2$ for all $k>0,$ so \eqref{as:kappa3} does not hold true. However, if we consider a convex combination of $\kappa_0^0$ with another kernel such as $\kappa_0^r$ with $r\in(0,1/2],$ then \eqref{as:kappatau} remains false for any $\gamma>0$ but \eqref{as:kappa3} is fulfilled. Indeed we have for $\rho\in(0,1)$
$$\int z^2(\rho\kappa_0^0(z)+(1-\rho)\kappa_0^r(z))\,dz=\f\rho2+\f{1-\rho}2(r^2+(1-r)^2)=\f12(1-2r(1-r)(1-\rho))<\f12.$$

\noindent{\bf Homogeneous fragmentation :}
\beq\label{ex:kappaal}\kappa_0^\al(z)=\f{\al+1}{2}(z^\al+(1-z)^\al)\qquad\text{for}\qquad\al>-1.\eeq
It gives another class of fragmentation kernels, namely in $L^1$ (unlike the mitosis case). The parameter $\gamma=1+\al>0$ suits for \eqref{as:kappatau} and so \eqref{as:kappa3} is fulfilled. It shows that our assumptions allow fragmentation at the ends of the polymers (called depolymerization, see \cite{Lenuzza}, when $\al$ is close to $-1$) once it is not the extreme case of renewal equation.\\
Uniform repartition ($\kappa(x,y)=\f1y\1_{0\leq x\leq y}$) corresponds to $\al=0$ and is also included.

\

This last case of uniform repartition is useful because it provides us with explicit formulas for the eigenelements. For instance, we can consider the two following examples.

\

\noindent{\bf First example :} $\,\tau(x)=\tau_0,\ \beta(x)=\beta_0x.$\\
In this case, widely used by \cite{Greer}, the eigenelements exist and we have
$$\lb=\sqrt{\beta_0\tau_0},$$
$$\U(x)=2\sqrt{\f{\beta_0}{\tau_0}}\Bigl(X+\f{X^2}2\Bigr)e^{-X-\f{X^2}2},\quad\text{with}\ X=\sqrt{\f{\beta_0}{\tau_0}}x,$$
$$\phi(x)=\f12(1+X).$$
\noindent{\bf Second example :} $\ \tau(x)=\tau_0x.$\\
For such $\beta$ for which there exists eigenelements, we have
$$\lb=\tau_0\qquad \text{and}\qquad \phi(x)=\f x{\int y\U(y)}.$$
For instance when $\beta(x)=\beta_0x^n$ with $n\in\N^*,$ then the eigenelements exist and we can compute $\U$ and $\phi$ and we have the formulas in Table \ref{tab:examples}. In this table we can notice that $\U(0)>0$ but the boundary condition $\tau\U(0)=0$ is fulfilled.

\begin{table}[h]
\begin{center}
\begin{tabular}{|c|c|c|c|}
\hline
$n=1$&$\lb=\tau_0$&$\U(x)=\f{\beta_0}{\tau_0}e^{-\f{\beta_0}{\tau_0}x}$&$\phi(x)=\f{\beta_0}{\tau_0}x$\\
\hline
$n=2$&$\lb=\tau_0$&$\U(x)=\sqrt{\f{2\beta_0}{\pi\tau_0}}e^{-\f12\f{\beta_0}{\tau_0}x^2}$&$\phi(x)=\sqrt{\f{\pi\beta_0}{2\tau_0}}x$\\
\hline
$n$ & $\lb=\tau_0$&$\U(x)=\Bigl(\f{\beta_0}{n\tau_0}\Bigr)^{\f1n}\f n{\Gamma(\f1n)}e^{-\f1n\f{\beta_0}{\tau_0}x^n}$&$\phi(x)=\Bigl(\f{\beta_0}{n\tau_0}\Bigr)^{\f1n}\f {\Gamma(\f1n)}{\Gamma(\f2n)}x$\\
\hline
\end{tabular}
\caption{\label{tab:examples}The example $\tau(x)=\tau_0x,\ \beta(x)=\beta_0x^n$ and uniform repartition $\kappa(x,y)=\f1y\1_{0\leq x\leq y}.$ The table gives the eigenelements solution to \eqref{eq:eigenproblem}.}
\end{center}
\end{table}

Now we turn to non-existence cases. Let us consider constant fragmentation $\beta(x)=\beta_0$ with an affine polymerization $\tau(x)=\tau_0+\tau_1x,$ and any fragmentation kernel $\kappa$ which satisfies to assumptions \eqref{as:kappa1}-\eqref{as:kappa2}. We notice that \eqref{as:betatauinf} is not satisfied and look at two instructive cases.

\

\noindent{\bf First case :} $\ \tau_0=0.$\\
In this case assumption \eqref{as:betatau0} does not hold true. Assume that there exists $\U\in L^1(\R^+)$ solution of \eqref{eq:eigenproblem} with the estimates of Theorem \ref{th:eigenelements}. Integrating the equation on $\U$ we obtain that $\lb=\beta_0,$ but multiplying the equation by $x$ before integration we have that $\lb=\tau_1.$ We conclude that eigenelements cannot exist if $\tau_1\neq\beta_0.$\\
Moreover, if we take $\kappa(x,y)=\f1y\1_{0\leq x\leq y},$ then a formal computation shows that any solution to the first equation of \eqref{eq:eigenproblem} belongs to the plan $Vect\{x^{-1},x^{-\f{2\beta_0}{\tau_1}}\}.$ So, even if $\beta_0=\tau_1$, there does not exist an eigenvector in $L^1.$

\

\noindent{\bf Second case :} $\ \tau_0>0.$\\
In this case \eqref{as:betatau0} holds true but the same integrations than before lead to
$$\int x\U(x)\,dx=\f{\tau_0}{\beta_0-\tau_1}.$$
So there cannot exist any eigenvector $\U\in L^1(x\,dx)$ for $\tau_1\geq\beta_0.$

\

\section{Proof of the main theorem}\label{se:proof}

The proof of Theorem \ref{th:eigenelements} is divided as follows. We begin with a result concerning the positivity of the $a\ priori$ existing eigenvectors (Lemma \ref{lm:positivity}). We then define, in Section \ref{subsec:truncated}, a regularized and truncated problem for which we know that eigenelements exist (see the Appendix \ref{se:KR} for a proof using the Krein-Rutman theorem), and we choose it such that the related eigenvalue is positive (Lemma \ref{lm:lambdapositivity}). In Section~\ref{subsec:estim}, we give a series of estimates that allow us to pass to the limit in the truncated problem and so prove the existence for the original eigenproblem \eqref{eq:eigenproblem}. The positivity of the eigenvalue $\lb$ and the uniqueness of the eigenelements are proved in the two last subsections. 

\subsection{A preliminary lemma}

Before proving Theorem \ref{th:eigenelements}, we give a preliminary lemma, useful to prove uniqueness of the eigenfunctions.

\begin{lemma}[Positivity]\label{lm:positivity}
Consider $\U$ and $\phi$ solutions to the eigenproblem \eqref{eq:eigenproblem}.\\
We define $\dis m:=\inf_{x,y}\bigl\{x\,:\,(x,y)\in Supp\,\beta(y)\kappa(x,y)\bigr\}.$ Then we have, under assumptions \eqref{as:kappa1}, \eqref{as:kappa2}, \eqref{as:taupositivity} and \eqref{as:betasupport}
$$Supp\,\U=[m,\infty)\qquad\text{and}\qquad\tau\U(x)>0\quad\forall x>m,$$
$$\phi(x)>0\quad\forall x>0.$$
If additionaly $\f1\tau\in L^1_0,$ then $\phi(0)>0.$
\end{lemma}

\begin{remark}
In case $Supp\,\kappa=\{(x,y)/x\leq y\},$ then $m=0$ and Lemma \ref{lm:positivity} and Theorem \ref{th:eigenelements} can be proved without the connexity condition \eqref{as:betasupport} on the support of $\beta.$
\end{remark}

\begin{proof}
Let $x_0>0,$ we define $F:x\mapsto \tau(x)\U(x)e^{\int_{x_0}^x\f{\lb+\beta(s)}{\tau(s)}ds}.$ We have that
\beq\label{Upositivity}F'(x)=2e^{\int_{x_0}^x\f{\lb+\beta(s)}{\tau(s)}ds}\int\beta(y)\kappa(x,y)\U(y)\,dy\geq0.\eeq
So, as soon as $\tau\U(x)$ once becomes positive, it remains positive for larger $x.$

\

We define $a:=\inf\{x\,:\,\tau(x)\U(x)>0\}.$ We first prove that $a\leq \f b2.$ For this we integrate the equation on $[0,a]$ to obtain
$$\int_0^a\int_a^\infty\beta(y)\kappa(x,y)\U(y)\,dydx=0,$$
$$\int_a^\infty\beta(y)\U(y)\int_0^a\kappa(x,y)\,dxdy=0.$$
Thus for almost every $y\geq\max(a,b),\ \int_0^a\kappa(x,y)\,dx=0.$ As a consequence we have
$$1=\int\kappa(x,y)\,dx=\int_a^y\kappa(x,y)\,dx\leq\f1a\int x\kappa(x,y)\,dx=\f y{2a}$$
thanks to \eqref{as:kappa1} and \eqref{as:kappa2}, and this is possible only if $b\geq2a.$

\

Assume by contradiction that $m<a,$ integrating \eqref{eq:eigenproblem} multiplied by $\varphi,$ we have for all $\varphi\in{\mathcal C}_c^\infty$ such that $Supp\,\varphi\subset[0,a]$
\beq\label{eq:positivity}\int\int\varphi(x)\beta(y)\kappa(x,y)\U(y)\,dydx=0.\eeq
By definition of $m$ and using the fact that $m<a,$ there exists $(p,q)\in(m,a)\times(b,\infty)$ such that $(p,q)\in Supp\,\beta(y)\kappa(x,y).$ But we can choose $\varphi$ positive such that $\varphi(p)\U(q)>0$ and this is a contradiction with \eqref{eq:positivity}.
So we have $m\geq a.$\\
To conclude we notice that on $[0,m],\ \U$ satisfies
$$\p_x(\tau(x)\U(x))+\lb\U(x)=0.$$
So, thanks to the condition $\tau(0)\U(0)=0$ and the assumption \eqref{as:taupositivity}, we have $\U\equiv0$ on $[0,m],$ so $m=a$ and the first statement is proved.

\

For $\phi,$ we define $G(x):=\phi(x)e^{-\int_{x_0}^x\f{\lb+\beta(s)}{\tau(s)}ds}.$ We have that
\beq\label{eq:phipositivity}G'(x)=-2e^{-\int_{x_0}^x\f{\lb+\beta(s)}{\tau(s)}ds}\beta(x)\int_0^x\kappa(y,x)\phi(y)\,dy\leq0,\eeq
so, as soon as $\phi$ vanishes, it remains null. Therefore $\phi$ is positive on an interval $(0,x_1)$ with $x_1\in\R_+^*\cup\{+\infty\}.$ Assuming that $x_1<+\infty$ and using that $x_1>a=m$ because $\int\phi(x)\U(x)dx=1,$ we can find $X\geq x_1$ such that
$$\int_{x_1}^XG'(x)\,dx=-2\int_{x_1}^X\int_0^{x_1}e^{\int_{x_0}^x\f{\lb+\beta(s)}{\tau(s)}ds}\phi(y)\beta(x)\kappa(y,x)\,dy\,dx<0.$$
This contradicts that $\phi(x)=0$ for $x\geq x_1,$ and we have proved that $\phi(x)>0$ for $x>0.$\\
If $\f1\tau\in L^1_0,$ we can take $x_0=0$ in the definition of $G$ and so $\phi(0)>0$ or $\phi\equiv0.$ The fact that $\phi$ is positive ends the proof of the lemma.

\qed
\end{proof}

\

\subsection{Truncated problem}
\label{subsec:truncated}
The proof of the theorem is based on uniform estimates on the solution to a truncated equation. Let $\eta,\ \delta,\ R$ positive numbers and define
$$\tau_\eta(x)=\left\{\begin{array}{ll}\eta&0\leq x\leq\eta\\
                                       \tau(x)&x\geq\eta.
                      \end{array}\right.$$
Then $\tau_\eta$ is lower bounded on $[0,R]$ thanks to \eqref{as:taupositivity} and we denote by $\mu=\mu(\eta,R):=\inf_{[0,R]}\tau_\eta.$
The existence of eigenelements $(\lb_\eta^\delta,\U_\eta^\delta,\phi_\eta^\delta)$ for the following truncated problem when $\delta R<\mu$ is standard (see Theorem \ref{th:KreinRutman} in the Appendix).

\beq\label{eq:truncated}
\left \{ \begin{array}{l}
\displaystyle \f{\p}{\p x} (\tau_\eta(x) \U_\eta^\delta(x)) + ( \beta(x) + \lb_{\eta}^\delta) \U_\eta^\delta(x) = 2 \int_x^R\beta(y)\kappa(x,y) \U_\eta^\delta(y)\,dy,\qquad 0<x<R,
\\
\\
\tau_\eta\U_\eta^\delta(x=0)=\delta,\qquad \U_\eta^\delta(x)>0 , \qquad \int \U_\eta^\delta(x)dx =1,
\\
\\
\displaystyle -\tau_\eta(x) \f{\p}{\p x} \phi_\eta^\delta(x) + ( \beta(x) + \lb_\eta^\delta) \phi_\eta^\delta(x) - 2 \beta(x) \int_0^x\kappa(y,x) \phi_\eta^\delta(y)\,dy = \delta\phi_\eta^\delta(0),\qquad 0<x<R,
\\
\\
\phi_\eta^\delta(R)=0, \qquad \phi_\eta^\delta(x)>0 , \qquad \int \phi_\eta^\delta(x)\U_\eta^\delta(x)dx =1.
\end{array} \right.
\eeq

The proof of the theorem \ref{th:eigenelements} requires $\lb_\eta^\delta>0.$ To enforce it, we take $\delta R=\f\mu2$ and we consider $R$ large enough to satisfy the following lemma.

\begin{lemma}\label{lm:lambdapositivity}
Under assumptions \eqref{as:kappa1}, \eqref{as:betatauspace} and \eqref{as:betatauinf}, there exists a $R_0>0$ such that for all $R>R_0,$ if we choose $\delta=\f\mu{2R},$ then we have $\lb_\eta^\delta>0.$
\end{lemma}

\begin{proof}
Assume by contradiction that $R>0$ and $\lb_\eta^\delta\leq0$ with $\delta=\f\mu{2R}.$ Then, integrating between $0$ and $x>0,$ we obtain
\begin{eqnarray*}
0&\geq&\lb\int_0^x\U(y)\,dy\\
&=&\delta-\tau(x)\U(x)-\int_0^x\beta(y)\U(y)\,dy+2\int_0^x\int_z^R\beta(y)\kappa(z,y)\U(y)\,dy\,dz\\
&=&\delta-\tau(x)\U(x)+\int_0^x\beta(y)\U(y)\,dy+2\int_x^R\Bigl(\int_0^x\kappa(z,y)\,dz\Bigr)\beta(y)\U(y)\,dy\\
&\geq&\delta-\tau(x)\U(x)+\int_0^x\beta(y)\U(y)\,dy.
\end{eqnarray*}
Consequently $$\tau(x)\U(x)\geq\delta+\int_0^x\f{\beta(y)}{\tau(y)}\tau(y)\U(y)\,dy$$
and, thanks to Gr\"onwall's lemma,
$$\tau(x)\U(x)\geq\delta e^{\int_0^x\f{\beta(y)}{\tau(y)}dy}.$$
But assumption \eqref{as:betatauinf} ensures that for all $n\geq0,$ there is a $A>0$ such that $$\f{\beta(x)}{\tau(x)}\geq\f nx,\qquad\forall x\geq A$$
and thus we have
$$\tau(x)\U(x)\geq\delta \Bigl(\f xA\Bigr)^n,\quad\forall x\geq A.$$
Due to Assumption~\eqref{as:betatauspace}, we can choose $n$ such that $x^{-n}\tau(x)\to0$ when $x\to+\infty.$
Then there exists $B>A$ such that $x^{-n}\tau(x)\leq \f\mu{4A^n}$ for $x\geq B,$ and we have, for $R>B,$
$$1=\int_0^R\U(x)\,dx\geq\int_B^R\U(x)\,dx\geq\delta\int_B^R\f{x^n}{A^n\tau(x)}\,dx\geq\f2R(R-B)$$
what is a contradiction as soon as $R>2B$; so Lemma \ref{lm:lambdapositivity} holds for $R_0=2B.$

\qed
\end{proof}

\

\subsection{Limit as $\delta\to0$ for $\U_\eta^\delta$ and $\lb_\eta^\delta$}
\label{subsec:estim}
Fix $\eta$ and let $\delta\rightarrow0$ (then $R\to\infty$ since $\delta R=\f\mu2$).

\paragraph{\it First estimate: $\lb_\eta^\delta$ upper bound.}
Integrating equation \eqref{eq:truncated} between $0$ and $R,$ we find
$$\lb_\eta^\delta\leq\delta+\int\beta(x)\U_\eta^\delta(x)\,dx,$$
then the idea is to prove a uniform estimate on $\int\beta\U_\eta^\delta.$ For this we begin with bounding the higher moments $\int x^\al\beta\U_\eta^\delta$ for $\al\geq\max{(2,\al_0+1)}:=m.$\\
Let $\al\geq m,$ according to \eqref{as:kappa3} we have
$$\int\f{x^\al}{y^\al}\kappa(x,y)\,dx\leq\int\f{x^2}{y^2}\kappa(x,y)\,dx\leq c<\f12.$$
Multiplying the equation on $\U_\eta^\delta$ by $x^\al$ and then integrating on $[0,R],$ we obtain for all $A\geq\eta$
\begin{eqnarray*}
\int x^\al\bigl((1-2c)\beta(x)\bigr)\U_\eta^\delta(x)\,dx&\leq&\al\int x^{\al-1}\tau_\eta(x)\U_\eta^\delta(x)\,dx\\
&=&\al\int_{x\leq A}x^{\al-1}\tau_\eta(x)\U_\eta^\delta(x)\,dx+\al\int_{x\geq A}x^{\al-1}\tau(x)\U_\eta^\delta(x)\,dx\\
&\leq&\al A^{\al-1-\al_0}\sup_{x\in(0,A)}{\{x^{\al_0}\tau(x)\}}+\omega_{A,\al}\int x^\al\beta(x)\U_\eta^\delta(x)\,dx,
\end{eqnarray*}
where $\omega_{A,\al}$ is a positive number chosen to have $\al\tau(x)\leq\omega_{A,\al} x\beta(x),\ \forall x\geq A.$ Thanks to \eqref{as:kappa3} and \eqref{as:betatauinf}, we can choose $A_\al$ large enough to have $\omega_{A_\al,\al}<1-2c.$
Thus we find
\beq\label{eq:L1bound1}\forall\al\geq m,\,\exists A_\al:\ \forall\eta,\delta>0,\quad\int x^\al\beta(x)\U_\eta^\delta(x)\,dx\leq\f{\al {A_\al}^{\al-1-\al_0}\sup_{(0,A)}{\{x^{\al_0}\tau(x)\}}}{1-2c-\omega_{A_\al,\al}}:=B_\al.\eeq
The next step is to prove the same estimates for $0\leq\al<m$ and for this we first give a bound on $\tau_\eta\U_\eta^\delta.$ We fix $\rho\in(0,1/2)$ and define $x_\eta>0$ as the unique point such that $\int_0^{x_\eta}\f{\beta(y)}{\tau_\eta(y)}dy=\rho.$ It exists because $\beta$ is nonnegative and locally integrable, and $\tau_\eta$ is positive. Thanks to assumption \eqref{as:betatau0}, we know that $x_\eta\underset{\eta\to0}{\longrightarrow}x_0$ where $x_0>0$ satisfies $\int_0^{x_0}\f{\beta(y)}{\tau(y)}dy=\rho,$ so $x_\eta$ is bounded by $0<\underline x\leq x_\eta\leq\overline x.$ Then, integrating \eqref{eq:truncated} between $0$ and $x\leq x_\eta,$ we find
\begin{eqnarray*}
\tau_\eta(x)\U_\eta^\delta(x)&\leq&\delta+2\int_0^{x}\int\beta(y)\U_\eta^\delta(y)\kappa(z,y)\,dy\,dz\\
&\leq&\delta+2\int\beta(y)\U_\eta^\delta(y)\,dy\\
&=&\delta+2\int_0^{x_\eta}\beta(y)\U_\eta^\delta(y)\,dy+2\int_{x_\eta}^\infty\beta(y)\U_\eta^\delta(y)\,dy\\
&\leq&\delta+2\sup_{(0,x_\eta)}\{\tau_\eta\U_\eta^\delta\}\int_0^{x_\eta}\f{\beta(y)}{\tau_\eta(y)}\,dy+\f2{x_\eta^m}\int_0^\infty y^m\beta(y)\U_\eta^\delta(y)\,dy\\
&\leq&\delta+2\rho\sup_{(0,x_\eta)}\{\tau_\eta\U_\eta^\delta\}+\f2{x_\eta^m}B_m.
\end{eqnarray*}
Consequently, if we consider $\delta\leq1$ for instance, we obtain
\beq\label{eq:Linfbound1}\sup_{x\in(0,\underline x)}\tau_\eta(x)\U_\eta^\delta(x)\leq\f{1+2B_m/\underline x^m}{1-2\rho}:=C\eeq
so $\tau_\eta\U_\eta^\delta$ is uniformly bounded in a neighborhood of zero.\\
Now we can prove a bound $B_\al$ for $x^\al\beta\U_\eta^\delta$ in the case $0\leq\al<m.$ Thanks to the estimates \eqref{eq:L1bound1} and \eqref{eq:Linfbound1} we have
\begin{eqnarray}\label{eq:L1bound2}
\int x^\al\beta(x)\U_\eta^\delta(x)\,dx&=&\int_0^{\overline x}x^\al\beta(x)\U_\eta^\delta(x)\,dx+\int_{\overline x}^Rx^\al\beta(x)\U_\eta^\delta(x)\,dx\nonumber\\
&\leq&\overline x^\al\sup_{(0,\overline x)}\{\tau_\eta\U_\eta^\delta\}\int_0^{\overline x}\f{\beta(y)}{\tau_\eta(y)}\,dy+\overline x^{\al-m}\int_{\overline x}^Rx^m\beta(x)\U_\eta^\delta(x)\,dx\nonumber\\
&\leq&C\rho\overline x^\al+B_m\overline x^{\al-m}:=B_\al.
\end{eqnarray}
Combining \eqref{eq:L1bound1} and \eqref{eq:L1bound2} we obtain
\beq\label{eq:L1bound3}\forall\al\geq0,\,\exists B_\al:\ \forall\eta,\delta>0,\quad\int x^\al\beta(x)\U_\eta^\delta(x)\,dx\leq B_\al,\eeq
and finally we bound $\lb_\eta^\delta$
\beq\label{eq:lambdaupperbound}\lb_\eta^\delta=\delta +\int\beta\U_\eta^\delta\leq\delta +B_0.\eeq
So the family $\{\lb_\eta^\delta\}_\delta$ belong to a compact interval and we can extract a converging subsequence \mbox{$\lb_\eta^\delta\underset{\delta\to0}{\longrightarrow}\lb_\eta.$}

\

\paragraph{\it Second estimate : $W^{1,1}$bound for $x^\al\tau_\eta\U_\eta^\delta,\ \al\geq0.$}
We use the estimate \eqref{eq:L1bound3}. First we give a $L^\infty$bound for $\tau_\eta\U_\eta^\delta$ by integrating \eqref{eq:truncated} between $0$ and $x$
\beq\label{eq:Linfbound2}\tau_\eta(x)\U_\eta^\delta(x)\leq\delta+2\int_0^R\beta(y)\U_\eta^\delta(y)\,dy\leq\delta+2B_0:=D_0.\eeq
Then we bound $x^\al\tau_\eta\U_\eta^\delta$ in $L^1$ for $\al>-1.$ Assumption \eqref{as:betatauinf} ensures that there exists $X>0$ such that\\
$\tau(x)\leq x\beta(x),\ \forall x\geq X,$ so we have for $R>X$
\begin{eqnarray*}
\int x^\al\tau_\eta(x)\U_\eta^\delta(x)\,dx&\leq&\sup_{(0,X)}\{\tau_\eta\U_\eta^\delta\}\int_0^X x^\al\,dx+\int_X^Rx^{\al+1}\beta(x)\U_\eta^\delta(x)\,dx\\
&\leq&\sup_{(0,X)}\{\tau_\eta\U_\eta^\delta\}\f{X^{\al+1}}{\al+1}+B_{\al+1}:=C_\al.
\end{eqnarray*}
Finally
\beq\label{eq:L1bound4}
\forall\al>-1,\,\exists C_\al:\ \forall\eta,\delta>0,\quad\int x^\al\tau_\eta(x)\U_\eta^\delta(x)\,dx\leq C_\al
\eeq
and we also have that $x^\al\U_\eta^\delta$ is bounded in $L^1$ because $\tau\in{\mathcal P}$ (see assumption \eqref{as:betatauspace}).\\
A consequence of \eqref{eq:L1bound3} and \eqref{eq:L1bound4} is that $x^\al\tau_\eta\U_\eta^\delta$ is bound in $L^\infty$ for all $\al\geq0.$ We already have \eqref{eq:Linfbound2} and for $\al>0$, we multiply \eqref{eq:truncated} by $x^\al,$ integrate on $[0,x]$ and obtain
$$
x^\al\tau_\eta(x)\U_\eta^\delta(x)\leq\al\int_0^R y^{\al-1}\tau_\eta(y)\U_\eta^\delta(y)\,dy+2\int_0^R y^\al\beta(y)\U_\eta^\delta(y)\,dy\leq\al C_\al+2B_\al:=D_\al,
$$
that give immediately
\beq\label{eq:Linfbound3}
\forall\al\geq0,\,\exists D_\al:\ \forall\eta,\delta>0,\quad\sup_{x>0}x^\al\tau_\eta(x)\U_\eta^\delta(x)\leq D_\al.
\eeq
To conclude we use the fact that neither the parameters nor $\U_\eta^\delta$ are negative and we find by the chain rule, for $\al\geq0$
$$\int\bigl|\f\p{\p x}(x^\al\tau_\eta(x)\U_\eta^\delta(x))\bigr|dx\leq\al\int x^{\al-1}\tau_\eta(x)\U_\eta^\delta(x)\,dx+\int x^\al\bigl|\p_x(\tau_\eta(x)\U_\eta^\delta(x))\bigr|\,dx\nonumber$$
\beq\label{eq:W11bound}\hspace{4cm}\leq\al\int x^{\al-1}\tau_\eta(x)\U_\eta^\delta(x)\,dx+\lb_\eta^\delta\int x^\al\U_\eta^\delta(x)\,dx+3\int x^\al\beta(x)\U_\eta^\delta(x)\,dx
\eeq
and all the terms in the right hand side are uniformly bounded thanks to the previous estimates.\\

\

Since we have proved that the family $\{x^\al\tau_\eta\U_\eta^\delta\}_\delta$ is bounded in $W^{1,1}(\R^+)$ for all $\al\geq0,$ then, because $\tau_\eta$ is positive and belongs to ${\mathcal P},$ we can extract from $\{\U_\eta^\delta\}_\delta$ a subsequence which converges in $L^1(\R^+)$ when $\delta\to0.$ Passing to the limit in equation \eqref{eq:truncated} we find that
\beq\label{eq:truncated2}\left\{\begin{array}{l}\dis\f{\p}{\p x} (\tau_\eta(x) \U_\eta(x))+(\beta(x)+\lb_{\eta})\U_\eta(x) = 2 \int_x^\infty\beta(y)\kappa(x,y) \U_\eta(y)\,dy,\\
\\
\U_\eta(0)=0,\quad\U_\eta(x)\geq0,\quad\int\U_\eta=1,\end{array}\right.\eeq
with $\lb_\eta\geq0.$

\

\subsection{Limit as $\eta\to0$ for $\U_\eta$ and $\lb_\eta$}

All the estimates \eqref{eq:L1bound1}-\eqref{eq:W11bound} remain true for $\delta=0.$ So we still know that the family $\{x^\al\tau_\eta\U_\eta\}_\eta$ belongs to a compact set of $L^1,$ but not necessarily $\{\U_\eta\}_\eta$ because in the limit $\tau$ can vanish at zero. We need one more estimate to study the limit $\eta\to0.$

\paragraph{\it Third estimate: $L^\infty$bound for $x^\al\tau_\eta\U_\eta,\ \al\geq-\gamma$.}
We already know that $x^\al\tau_\eta\U_\eta$ is bounded for $\al\geq0.$ So, to prove the bound, it only remains to prove that $x^{-\gamma}\tau_\eta\U_\eta$ is bounded in a neighborhood of zero. Let define $f_\eta:x\mapsto\sup_{(0,x)}\tau_\eta\U_\eta.$ If we integrate \eqref{eq:truncated2} between $0$ and $x'<x,$ we find
$$\tau_\eta(x')\U_\eta(x')\leq2\int_0^{x'}\int\beta(y)\U_\eta(y)\kappa(z,y)\,dy\,dz\leq2\int_0^x\int\beta(y)\U_\eta(y)\kappa(z,y)\,dy\,dz$$
and so for all $x$
$$f_\eta(x)\leq2\int_0^x\int\beta(y)\U_\eta(y)\kappa(z,y)\,dy\,dz.$$
We consider $x_\eta$ and $\underline x$ defined in the first estimate and, using \eqref{as:kappatau} and \eqref{as:betatau0}, we have for all $x<x_\eta$
\begin{eqnarray*}
f_\eta(x)&\leq&2\int_0^x\int\beta(y)\U_\eta(y)\kappa(z,y)\,dy\,dz\\
&=&2\int\beta(y)\U_\eta(y)\int_0^x\kappa(z,y)\,dz\,dy\\
&\leq&2\int_0^\infty\beta(y)\U_\eta(y)\min\Bigl(1,C\Bigl(\f x y\Bigr)^\gamma\Bigr)\,dy\\
&=&2\int_0^x\beta(y)\U_\eta(y)\,dy+2C\int_x^{x_\eta}\beta(y)\U_\eta(y)\Bigl(\f x y\Bigr)^\gamma \,dy+2C\int_{x_\eta}^\infty\beta(y)\U_\eta(y)\Bigl(\f x y\Bigr)^\gamma \,dy\\
&=&2\int_0^x\f{\beta(y)}{\tau_\eta(y)}\tau_\eta(y)\U_\eta(y)\,dy+2Cx^\gamma\int_x^{x_\eta}\f{\beta(y)}{\tau_\eta(y)}\f{\tau_\eta(y)\U_\eta(y)}{y^\gamma}\,dy+2C\int_{x_\eta}^\infty\beta(y)\U_\eta(y)\Bigl(\f x y\Bigr)^\gamma \,dy\\
&\leq&2f_\eta(x)\int_0^{x_\eta}\f{\beta(y)}{\tau_\eta(y)}\,dy+2Cx^\gamma\int_x^{x_\eta}\f{\beta(y)}{\tau_\eta(y)}\f{f_\eta(y)}{y^\gamma}\,dy+2C\|\beta\U_\eta\|_{L^1}\f{x^\gamma}{x_\eta^\gamma}.
\end{eqnarray*}
We set ${\mathcal V}_\eta(x)=x^{-\gamma}f_\eta(x)$ and we obtain
$$(1-2\rho){\mathcal V}_\eta(x)\leq K+2C\int_x^{x_\eta}\f{\beta(y)}{\tau_\eta(y)}{\mathcal V}_\eta(y)\,dy.$$
Hence, using Gr\"onwall's lemma, we find that $\dis{\mathcal V}_\eta(x)\leq\f {Ke^{\f{2C\rho}{1-2\rho}}}{1-2\rho}$ and consequently \beq\label{eq:0Linfbound}x^{-\gamma}\tau_\eta(x)\U_\eta(x)\leq\f {Ke^{\f{2C\rho}{1-2\rho}}}{1-2\rho}:=\wt C,\quad\forall x\in[0,\underline x].\eeq

\

This last estimate allows us to bound $\U_\eta$ by $\f{x^\gamma}{\tau}$ which is in $L^1_0$ by the assumption \eqref{as:kappatau}. Thanks to the second estimate, we also have that $\int x^\al\U_\eta$ is bounded in $L^1$ and so, thanks to the Dunford-Pettis theorem (see \cite{Brezis} for instance), $\{\U_\eta\}_\eta$ belong to a $L^1$-weak compact set. Thus we can extract a subsequence which converges $L^1-$weak toward $\U.$ But for all $\e>0,\ \{x^\al\U_\eta\}_\eta$ is bounded in $W^{1,1}([\e,\infty))$ for all $\al\geq1$ thanks to \eqref{eq:W11bound} and so the convergence is strong on $[\e,\infty).$ Then we write
\begin{eqnarray*}
\int|\U_\eta-\U|&=&\int_0^\e|\U_\eta-\U|+\int_\e^\infty|\U_\eta-\U|\\
&\leq&2\wt C\int_0^\e\f{x^\gamma}{\tau(x)}+\int_\e^\infty|\U_\eta-\U|.
\end{eqnarray*}
The first term on the right hand side is small for $\e$ small because $\f{x^\gamma}{\tau}\in L^1_0$ and then the second term is small for $\eta$ small because of the strong convergence. Finally $\U_\eta\underset{\eta\to0}{\longrightarrow}\U$ strongly in $L^1(\R^+)$ and $\U$ solution of the eigenproblem \eqref{eq:eigenproblem}.

\

\subsection{Limit as $\delta,\eta\to0$ for $\phi_\eta^\delta$}

We prove uniform estimates on $\phi_\eta^\delta$ which are enough to pass to the limit and prove the result.

\paragraph{\it Fourth estimate : uniform $\phi_\eta^\delta$-bound on $[0,A]$.}
Let $A>0,$ our first goal is to prove the existence of a constant $C_0(A)$ such that
$$\forall\eta,\delta,\qquad\sup_{(0,A)}{\phi_\eta^\delta}\leq C_0(A).$$
We divide the equation on $\phi_\eta^\delta$ by $\tau_\eta$ and we integrate between $x$ and $x_\eta$ with $0<x<x_\eta,$ where $x_\eta,$ bounded by $\underline x$ and $\overline x,$ is defined in the first estimate. Considering $\delta<\f{\mu(1-2\rho)}{\overline x}$ (fulfilled for $R>\f{\overline x}{2(1-2\rho)}$ since $\delta=\f{\mu}{2R}$), we find
\begin{eqnarray*}
\phi_\eta^\delta(x)&\leq&\phi_\eta^\delta(x_\eta)+2\int_x^{x_\eta}\f{\beta(y)}{\tau_\eta(y)}\int_0^y\kappa(z,y)\phi_\eta^\delta(z)\,dz+x_\eta\f\delta \mu\phi_\eta^\delta(0)\\
&\leq&\phi_\eta^\delta(x_\eta)+\sup_{(0,x_\eta)}\{\phi_\eta^\delta\}\Bigl(2\int_0^{x_\eta}\f{\beta(y)}{\tau_\eta(y)}\int_0^y\kappa(z,y)\,dz+x_\eta\f\delta\mu\Bigr)
\end{eqnarray*}
and we obtain
$$\sup_{x\in(0,\underline x)}{\phi_\eta^\delta(x)}\leq\f 1 {1-2\rho-\delta\overline x/\mu}\phi_\eta^\delta(x_\eta).$$
Using the decay of $\phi_\eta^\delta(x)e^{-\int_{\underline x}^x\f{\beta+\lb_\eta^\delta}{\tau_\eta}},$ there exists $C(A)$ such that
$$\sup_{x\in(0,A)}{\phi_\eta^\delta(x)}\leq C(A)\phi_\eta^\delta(x_\eta).$$
Noticing that $\int\phi_\eta^\delta(x)\U_\eta^\delta(x)dx =1,$ we conclude
$$1\geq\int_0^{x_\eta}\phi_\eta^\delta(x)\U_\eta^\delta(x)dx\geq \phi_\eta^\delta(x_\eta)\int_0^{x_\eta}e^{-\int_x^{x_\eta}\f{\beta+\lb_\eta^\delta}{\tau_\eta}}\U_\eta^\delta(x)\,dx,$$
so, as $x_\eta\to x_0$ and $\int_0^{x_0}\U(x)dx>0$ (thanks to Lemma \ref{lm:positivity} and because $x_0>b\geq a$), we have
\beq\label{eq:phi0bound}\sup_{(0,A)}{\phi_\eta^\delta}\leq C_0(A).\eeq

\

\paragraph{\it Fifth estimate : uniform $\phi_\eta^\delta$-bound on $[A,\infty)$.}
Following an idea introduced in \cite{PR} we notice that the equation in \eqref{eq:truncated} satisfied by $\phi_\eta^\delta$ is a transport equation and therefore satisfies the maximum principle (see Lemma \ref{lm:supersolution} in the Appendix). Therefore it remains to build a supersolution $\overline\phi$ that is positive at $x=R,$ to conclude $\phi_\eta^\delta(x)\leq\overline\phi(x)$ on $[0,R].$

This we cannot do on $[0,R],$ but on a subinterval $[A_0,R]$ only. So we begin with an auxiliary function $\overline\vp(x)=x^k+\theta$ with $k$ and $\theta$ positive numbers to be determined. We have to check that on $[A_0,R]$
$$-\tau(x)\f{\p}{\p x}\overline\vp(x)+(\lb_\eta^\delta+\beta(x))\overline\vp(x)\geq 2\beta(x)\int\kappa(y,x)\overline\vp(y)\,dy+\delta\phi_\eta^\delta(0),$$
{\it i.e.}
$$-k\tau(x)x^{k-1}+(\lb_\eta^\delta+\beta(x))\overline\vp(x)\geq\Bigl(2\theta+2\int\kappa(y,x)y^k\,dy\Bigr)\beta(x)+\delta\phi_\eta^\delta(0).$$
For $k\geq2,$ we know that $\int\kappa(y,x)\f{y^k}{x^k}\,dy\leq c<1/2$ so it is sufficient to prove that there exists $A_0>0$ such that we have
\beq\label{eq:sursolution1}-k\tau(x)x^{k-1}+(\lb_\eta^\delta+\beta(x))(x^k+\theta)\geq(2\theta+2cx^k)\beta(x)+\delta C_0(1)\eeq
for all $x>A_0,$ where $C_0$ is defined in \eqref{eq:phi0bound}.
For this, dividing \eqref{eq:sursolution1} by $x^{k-1}\tau(x),$ we say that if we have
\beq\label{eq:sursolution2}(1-2c)\f{x\beta(x)}{\tau(x)}\geq k+\f{2\theta\beta(x)+\delta C_0(1)}{x^{k-1}\tau(x)},\eeq
then \eqref{eq:sursolution1} holds true. Thanks to assumptions \eqref{as:betatauspace} and \eqref{as:betatauinf} we know that there exists $k>0$ such that for any $\theta>0,$ there exists $A_0>0$ for which \eqref{eq:sursolution2} is true on $[A_0,+\infty).$

Then we conclude by choosing the supersolution $\overline\phi(x)=\f{C_0(A_0)}\theta\overline\vp(x)$ so that
$$\overline\phi(x)\geq\phi_\eta^\delta(x)\quad \text{on} \ [0,A_0],$$
and on $[A_0,R],$ we have
\beq\label{eq:supersol}\left\{\begin{array}{l}
-\tau(x)\f{\p}{\p x}\overline\phi(x)+(\lb_\eta^\delta+\beta(x))\overline\phi(x) \geq 2\beta(x)\int_0^x\kappa(y,x)\overline\phi(y)\,dy+\delta\phi_\eta^\delta(0),
\\
\\
\overline\phi(R)>0,
\end{array}\right.\eeq
which is a supersolution to the equation satisfied by $\phi_\eta^\delta.$ Therefore $\phi_\eta^\delta\leq\overline\phi$ uniformly in $\eta$ and $\delta$ and we get
\beq\label{eq:phibound}\exists k,\theta,C\ s.t.\ \forall\eta,\delta,\quad\phi_\eta^\delta(x)\leq(Cx^k+\theta).\eeq

\

Equation \eqref{eq:truncated} and the fact that $\phi_\eta^\delta$ is uniformly bounded in $L^\infty_{loc}(\R^+)$ give immediately that $\p_x\phi_\eta^\delta$ is uniformly bounded in $L^\infty_{loc}(\R^+,\tau(x)dx),$ so in $L^\infty_{loc}(0,\infty)$ thanks to \eqref{as:taupositivity}.

\

Then we can extract a subsequence of $\{\phi_\eta^\delta\}$ which converges ${\mathcal C}^0(0,\infty)$ toward $\phi.$ Now we check that $\phi$ satisfied the adjoint equation of \eqref{eq:eigenproblem}. We consider the terms of \eqref{eq:truncated} one after another.\\
First $(\lb_\eta^\delta+\beta(x))\phi_\eta^\delta(x)$ converges to $(\lb+\beta(x))\phi(x)$ in $L^\infty_{loc}.$\\
For $\p_x\phi_\eta^\delta,$ we have an $L^\infty$ bound on each compact of $(0,\infty).$ So it converges $L^\infty-*weak$ toward $\p_x\phi.$\\
It remains the last term which we write, for all $x>0,$ $$\int_0^x\kappa(y,x)(\phi_\eta^\delta(y)-\phi(y))\,dy\leq\|\phi_\eta^\delta-\phi\|_{L^\infty(0,x)}\underset{\eta,\delta\to0}{\longrightarrow}0.$$
The fact that $\int\phi\U=1$ comes from the convergence $L^\infty-L^1$ when written as
$$1=\int\phi_\eta^\delta(x)\U_\eta^\delta(x)\,dx=\int\f{\phi_\eta^\delta(x)}{1+x^k}(1+x^k)\U_\eta^\delta(x)\,dx\longrightarrow\int\f{\phi(x)}{1+x^k}(1+x^k)\U(x)\,dx=\int\phi\U.$$

\bigskip\bigskip

At this stage we have found $(\lb,\U,\phi)\in\R^+\times L^1(\R^+)\times{\mathcal C}(\R^+)$ solution of \eqref{eq:eigenproblem}. The estimates announced in Theorem \ref{th:eigenelements} also follow from those uniform estimates. It remains to prove that $\lb>0$ and the uniqueness.

\

\subsection{Proof of $\lb>0$}

We prove a little bit more, namely that
\beq\label{eq:lowerbound}\lb\geq\f12\sup_{x\geq0}\{\tau(x)\U(x)\}. \eeq

We integrate the first equation of \eqref{eq:eigenproblem} between $0$ and $x$ and find
\begin{eqnarray*}
0\leq\lb\int_0^x\U(y)\,dy&=&-\tau(x)\U(x)-\int_0^x\beta(y)\U(y)\,dy+2\int_0^x\int_z^\infty\beta(y)\kappa(z,y)\U(y)\,dy\,dz\\
 &\leq&-\tau(x)\U(x)+2\int_0^\infty\int_z^\infty\beta(y)\kappa(z,y)\U(y)\,dy\,dz\\
 &=&-\tau(x)\U(x)+2\int_0^\infty\beta(y)\U(y)\,dy\\
 &=&-\tau(x)\U(x)+2\lb,
\end{eqnarray*}
Hence $2\lb\geq \tau(x)\U(x)$ and \eqref{eq:lowerbound} is proved.

\

\subsection{Uniqueness}

We follow the idea of \cite{M1}. Let $(\lb_1,\U_1,\phi_1)$ and $(\lb_2,\U_2,\phi_2)$ two solutions to the eigenproblem \eqref{eq:eigenproblem}.
First we have
\begin{eqnarray*}
\lb_1\int\U_1(x)\phi_2(x)\,dx&=&\int\Bigl(-\p_x(\tau(x)\U_1(x))-\beta(x)\U_1(x)+2\int_x^\infty\beta(y)\kappa(x,y)\U_1(y)\,dy\Bigr)\phi_2(x)\,dx\\
&=&\int\Bigl(\tau(x)\p_x\phi_2(x)-\beta(x)\phi_2(x)+2\beta(x)\int_0^x\kappa(y,x)\phi_2(y)\,dy\Bigr)\U_1(x)\,dx\\
&=&\lb_2\int\U_1(x)\phi_2(x)\,dx
\end{eqnarray*}
and then $\lb_1=\lb_2=\lb$ because $\int\U_1\phi_2>0$ thanks to Lemma \ref{lm:positivity}. \\
For the eigenvectors we use the General Relative Entropy method introduced in \cite{MMP1,MMP2}. For $C>0,$ we test the equation on $\U_1$ against $\sgn\bigl(\f{\U_1}{\U_2}-C\bigr)\phi_1,$
$$0=\int\Bigl[\p_x(\tau(x)\U_1(x))+(\lb+\beta(x))\U_1(x)-2\int_x^\infty\beta(y)\kappa(x,y)\U_1(y)\,dy\,\Bigr]\sgn\Bigl(\f{\U_1}{\U_2}(x)-C\Bigr)\phi_1(x)\,dx.$$
Deriving $\Bigl|\f{\U_1}{\U_2}(x)-C\Bigr|\tau(x)\U_2(x)\phi_1(x)$ we find
$$\begin{array}{l}\dis
\int\p_x(\tau(x)\U_1(x))\sgn\Bigl(\f{\U_1}{\U_2}(x)-C\Bigr)\phi_1(x)\,dx=\int\p_x\Bigl(\Bigl|\f{\U_1}{\U_2}(x)-C\Bigr|\tau(x)\U_2(x)\phi_1(x)\Bigr)\,dx\\
\\
\dis\hspace{1cm}+\int\p_x(\tau(x)\U_2(x))\f{\U_1}{\U_2}(x)\sgn\Bigl(\f{\U_1}{\U_2}(x)-C\Bigr)\phi_1(x)\,dx-\int\Bigl|\f{\U_1}{\U_2}(x)-C\Bigr|\p_x(\tau(x)\U_2(x)\phi_1(x))\,dx
\end{array}$$

and then

$$\begin{array}{l}\dis
\int\p_x(\tau(x)\U_1(x))\sgn\Bigl(\f{\U_1}{\U_2}(x)-C\Bigr)\phi_1(x)\,dx=\\
\\
\dis\hspace{2cm}2\int\Bigl|\f{\U_1}{\U_2}(x)-C\Bigr|\Bigl[\int_0^x\beta(x)\kappa(y,x)\U_2(x)\phi_1(y)\,dy-\int_x^\infty\beta(y)\kappa(x,y)\U_2(y)\phi_1(x)\,dy\Bigr]\,dx\\
\\
\dis\hspace{4cm}+2\int\int_x^\infty\beta(y)\kappa(x,y)\U_2(y)\,dy\,\f{\U_1}{\U_2}(x)\sgn\Bigl(\f{\U_1}{\U_2}(x)-C\Bigr)\phi_1(x)\,dx\\
\\
\dis\hspace{6cm}-\int(\lb+\beta(x))\f{\U_1}{\U_2}(x)\sgn\Bigl(\f{\U_1}{\U_2}(x)-C\Bigr)\U_2(x)\phi_1(x)\,dx,
\end{array}$$

\

$$\begin{array}{l}\dis
\int\p_x(\tau(x)\U_1(x))\sgn\Bigl(\f{\U_1}{\U_2}(x)-C\Bigr)\phi_1(x)\,dx=\\
\\
\dis\hspace{2cm}2\int\int\beta(y)\kappa(x,y)\Bigl[\Bigl|\f{\U_1}{\U_2}(y)-C\Bigr|-\Bigl|\f{\U_1}{\U_2}(x)-C\Bigr|\Bigr]\U_2(y)\phi_1(x)\,dxdy\\
\\
\dis\hspace{4cm}+2\int\int_x^\infty\beta(y)\kappa(x,y)\U_2(y)\,dy\,\f{\U_1}{\U_2}(x)\sgn\Bigl(\f{\U_1}{\U_2}(x)-C\Bigr)\phi_1(x)\,dx\\
\\
\dis\hspace{6cm}-\int(\lb+\beta(x))\f{\U_1}{\U_2}(x)\sgn\Bigl(\f{\U_1}{\U_2}(x)-C\Bigr)\U_2(x)\phi_1(x)\,dx.
\end{array}$$

\

So
$$\begin{array}{l}\dis
0=2\int\int\beta(y)\kappa(x,y)\Bigl[\Bigl|\f{\U_1}{\U_2}(y)-C\Bigr|-\Bigl|\f{\U_1}{\U_2}(x)-C\Bigr|\Bigr]\U_2(y)\phi_1(x)\,dxdy\\
\\
\dis\hspace{4cm}+2\int\int_x^\infty\beta(y)\kappa(x,y)\U_2(y)\,dy\,\f{\U_1}{\U_2}(x)\sgn\Bigl(\f{\U_1}{\U_2}(x)-C\Bigr)\phi_1(x)\,dx\\
\\
\dis\hspace{6cm}-2\int\int_x^\infty\beta(y)\kappa(x,y)\U_1(y)\,dy\,\sgn\Bigl(\f{\U_1}{\U_2}(x)-C\Bigr)\phi_1(x)\,dx
\end{array}$$

\

$$0=\int\int\beta(y)\kappa(x,y)\U_2(y)\Bigl|\f{\U_1}{\U_2}(y)-C\Bigr|\Bigl[1-\sgn\Bigl(\f{\U_1}{\U_2}(x)-C\Bigr)\sgn\Bigl(\f{\U_1}{\U_2}(y)-C\Bigr)\Bigr]\phi_1(x)\,dxdy.$$

\

Hence $\Bigl[1-\sgn\Bigl(\f{\U_1}{\U_2}(x)-C\Bigr)\sgn\Bigl(\f{\U_1}{\U_2}(y)-C\Bigr)\Bigr]=0$ on the support of $\kappa(x,y)$ for all $C$ thus $\f{\U_1}{\U_2}(x)=\f{\U_1}{\U_2}(y)$ on the support of $\kappa(x,y)$ and
\beq\label{eq:uniqueness}\p_x\f{\U_1}{\U_2}(x)=\int\beta(y)\kappa(x,y)\Bigl(\f{\U_1}{\U_2}(y)-\f{\U_1}{\U_2}(x)\Bigr)\f{\U_2(y)}{\U_2(x)}\,dy=0\eeq
so $\dis\f{\U_1}{\U_2}\equiv cst=1.$

\

We can prove in the same way that $\phi_1=\phi_2$ even if we can have $\U\equiv0$ on $[0,m]$ with $m>0.$ Indeed in this case we know that $\beta\equiv0$ on $[0,m]$ and so
$$\phi_i(x)=\phi_i(0)e^{\int_0^x\f{\lb}{\tau(s)}ds}\quad\forall x\in[0,m],\ i\in\{1,2\}.$$

\

\section{Conclusion, Perspectives}\label{se:csq}

We have proved the existence and uniqueness of eigenelements for the aggregation-fragmentation equation \eqref{eq:temporel} with assumptions on the parameters as large as possible, in order to render out the widest variety of biological or physical models. It gives access to the asymptotic behaviour of the solution by the use of the General Relative Entropy principle. 

\

A following work is to study the dependency of the eigenvalue $\lb$ on parameters $\tau$ and $\beta$ (see \cite{M2}). For instance, our assumptions allow $\tau$ to vanish at zero, what is a necessary condition to ensure that $\lb$ tends to zero when the fragmentation tends to infinity. Such results give precious information on the qualitative behaviour of the  solution.

\

Another possible extension of the present work is to prove existence of eigenelements in the case of time-periodic parameters, using the Floquet's theory, and then compare the new $\lb_F$ with the time-independent one $\lb$ (see \cite{Lepoutre}). Such studies can help to choose a right strategy in order to optimize, for instance, the total mass $\int x u(t,x) dx$  in the case of prion proliferation (see \cite{CL1}) or on the contrary minimize the total population $\int u(t,x) dx$ in the case of cancer therapy (see \cite{Lepoutre, Clairambault}).

\

Finally, this eigenvalue problem could be used to recover some of the equation parameters like $\tau$ and $\beta$ from the knowledge of the asymptotic profile of the solution, as introduced in \cite{DPZ, PZ} in the case of symmetric division ($\tau=1$ and $\kappa=\delta_{x=\frac{y}{2}}$), by the use of inverse problems techniques. The method of \cite{PZ} has to be adapted to our general case, in order to model prion proliferation for instance, or yet to recover the aggregation rate $\tau$ ; this is another direction for future research.

\vspace{1cm}

{\bf Aknowledgment}

The authors thank a lot Beno\^it Perthame for his precious help and his corrections.

\newpage{\noindent\LARGE\bf Appendix}

\appendix

\section{Assumption on $\kappa.$}

\begin{lemma}\label{lm:kappa}
Assumptions \eqref{as:kappa1},\eqref{as:kappa2} and \eqref{as:kappatau} with $\gamma>0$ imply that
$$\inf_y\ \lim_{\eta\to0}\int_{\eta y}^{(1-\eta)y}\kappa(x,y)\,dx>0,$$
which means that polymers undergo a decrease in the size during fragmentation process. As a consequence, assumption \eqref{as:kappa3} holds true.
\end{lemma}

\begin{proof}
With the first assumption \eqref{as:kappa1} we have
$$1=\int_0^y\kappa(x,y)\,dx=\int_0^{\eta y}\kappa(x,y)\,dx+\int_{\eta y}^{(1-\eta)y}\kappa(x,y)\,dx+\int_{(1-\eta)y}^y\kappa(x,y)\,dx.$$
The two other assumptions \eqref{as:kappa2} and \eqref{as:kappatau} allow to control the mass of $\kappa$ at the ends :
\begin{eqnarray*}
\int_{\eta y}^{(1-\eta)y}\kappa(x,y)\,dx&=&1-\int_0^{\eta y}\kappa(x,y)\,dx-\int_{(1-\eta)y}^y\kappa(x,y)\,dx\\
 &\geq&1-C\eta^\gamma-\f1{1-\eta}\int_{(1-\eta)y}^y\f xy\kappa(x,y)\,dx\\
 &\geq&1-C\eta^\gamma-\f1{2(1-\eta)}\quad\underset{\eta\to0}{\longrightarrow}\quad\f12,
\end{eqnarray*}
which gives the first assertion of the lemma.
\\
Now we can prove \eqref{as:kappa3} :
\begin{eqnarray*}
\int_0^y\f{x^2}{y^2}\kappa(x,y)\,dx&\leq&\Bigr[\int_0^{\eta y}\f xy\kappa(x,y)\,dx+\int_{(1-\eta)y}^y\f xy\kappa(x,y)\,dx\Bigl]+\int_{\eta y}^{(1-\eta)y}\f{x^2}{y^2}\kappa(x,y)\,dx\\
 &\leq&\Bigr[\f12-\int_{\eta y}^{(1-\eta)y}\f xy\kappa(x,y)\,dx\Bigl]+(1-\eta)\int_{\eta y}^{(1-\eta)y}\f xy\kappa(x,y)\,dx\\
 &=&\f12-\eta\int_{\eta y}^{(1-\eta)y}\f xy\kappa(x,y)\,dx\\
 &\leq&\f12-\eta^2\int_{\eta y}^{(1-\eta)y}\kappa(x,y)\,dx.
\end{eqnarray*}
We use the first part of the proof to conclude. Taking $\eta=\min\Bigl(\f14,\f1{(4C)^{1/\gamma}}\Bigr)$ for instance, we obtain
$$\int_{\eta y}^{(1-\eta)y}\kappa(x,y)\,dx\geq\f13,$$
and the lemma is proved for $c=\f12-\f1{48}.$
\qed
\end{proof}

\

\section{Krein-Rutman}\label{se:KR}

We prove existence of solution for the truncated equation \eqref{eq:truncated}. In this part $\eta$ and $\delta$ are fixed (with $\delta R<\mu$), so we will omit these indices for $\tau,\ \lb,\ \U$ and $\phi$ but we keep in mind that $\tau(x)\geq\mu>0.$\\
We use the Krein-Rutman theorem which requires working in the space of continuous functions (see \cite{BP} for instance). First we define regularized parameters as follows :
$$\tau_\e=\rho_\e\ast\tau,\quad\beta_\e=\rho_\e\ast\beta,\quad\text{and}\quad\forall y\geq0,\ \kappa_\e(.,y)=\rho_\e\ast\kappa(.,y),$$
where $\rho_\e(x)=\f1\e\rho(\f x\e)$ with $\rho\in{\mathcal C}_c^\infty((0,\infty)),$ positive and such that $\int_0^\infty\rho=1.$ Then we have the theorem
\begin{theorem}\label{th:KreinRutman}
Under assumptions \eqref{as:kappa1}-\eqref{as:betatauinf} on the parameters and for all $\e>0,$ there is a unique solution $\lb_\e\in\R$ and $\U_\e,\phi_\e\in{\mathcal C}^1([0,R])$ to the regularized eigenproblem
\beq\label{eq:regularized}
\left \{ \begin{array}{l}
\displaystyle \f{\p}{\p x} (\tau_\e(x) \U_\e(x)) + ( \beta_\e(x) + \lb_\e) \U_\e(x) = 2 \int_0^R\beta_\e(y)\kappa_\e(x,y) \U_\e(y)\,dy,\qquad 0<x<R,
\\
\\
\tau_\e\U_\e(x=0)=\delta\int_0^R\U_\e(y)\,dy,\qquad \U_\e(x)>0 , \qquad \int_0^R \U_\e(x)dx =1,
\\
\\
\displaystyle -\tau_\e(x) \f{\p}{\p x} \phi_\e(x) + ( \beta_\e(x) + \lb_\e) \phi_\e(x) - 2 \beta_\e(x) \int_0^R\kappa_\e(y,x) \phi_\e(y)\,dy = \tau_\e(0)\delta\phi_\e(0),\quad0<x<R,
\\
\\
\phi_\e(R)=0, \qquad \phi_\e(x)>0 , \qquad \int_0^R \phi_\e(x)\U_\e(x)dx =1.
\end{array} \right.
\eeq
\end{theorem}

\begin{proof}
We follow the proof of \cite{BP}. We define linear operators on $E:={\mathcal C}^0([0,R])$ to apply the Krein-Rutman theorem.

\paragraph{Direct equation.}
For $\nu>0$ we consider the following equation on $E$
\beq\label{eq:defopA}\left\{\begin{array}{l}
\dis\f{\p}{\p x} (n(x)) + \f{\nu+\beta_\e(x)}{\tau_\e(x)} n(x) - 2 \int_0^R\f{\beta_\e(y)}{\tau_\e(y)}\kappa_\e(x,y) n(y)\,dy = \f{f(x)}{\tau_\e(x)},\qquad 0\leq x\leq R,
\\
\\
n(x=0)=\delta\int_0^R\f{n(y)}{\tau_\e(y)}dy,\end{array}\right.\eeq
and we prove that the linear operator $A:f\mapsto n$ (solution of \eqref{eq:defopA}) satisfies to the assumptions of the Krein-Rutman theorem.

\paragraph{\it First step: construction of A.}
Fix $f\in E$ and for $m\in E,$ we define $n=T(m)\in E$ as the (explicit) solution to
$$\left\{\begin{array}{l}
\dis\f{\p}{\p x} (n(x)) + \f{\nu+\beta_\e(x)}{\tau_\e(x)} n(x) = 2 \int_0^R\f{\beta_\e(y)}{\tau_\e(y)}\kappa_\e(x,y) m(y)\,dy + \f{f(x)}{\tau_\e(x)},\qquad 0\leq x\leq R,
\\
\\
n(x=0)=\delta\int_0^R\f{m(y)}{\tau_\e(y)}dy,\end{array}\right.$$
We prove that $T$ is a strict contraction. Therefore it has a unique fixed point thanks to the Banach-Picard theorem. This fixed point is a solution to \eqref{eq:defopA}.

In order to prove that $T$ is a strict contraction, we consider $m_1$ and $m_2$ two functions in $E,$ we compute for $n=n_1-n_2,\ m=m_1-m_2,$
$$\left\{\begin{array}{l}
\dis\f{\p}{\p x} (n(x)) + \f{\nu+\beta_\e(x)}{\tau_\e(x)} n(x) = 2 \int_0^R\f{\beta_\e(y)}{\tau_\e(y)}\kappa_\e(x,y) m(y)\,dy ,\qquad 0\leq x\leq R,
\\
\\
n(x=0)=\delta\int_0^R\f{m(y)}{\tau_\e(y)}dy,\end{array}\right.$$
therefore
$$\left\{\begin{array}{l}
\dis\f{\p}{\p x} |n(x)| + \f{\nu+\beta_\e(x)}{\tau_\e(x)} |n(x)| \leq 2 \int_0^R\f{\beta_\e(y)}{\tau_\e(y)}\kappa_\e(x,y) |m(y)|\,dy ,\qquad 0\leq x\leq R,
\\
\\
n(x=0)\leq\delta\int_0^R\f{|m(y)|}{\tau_\e(y)}dy.\end{array}\right.$$
After integration, we obtain
$$|n(x)|e^{\int_0^x\f{\mu+\beta_\e}{\tau_\e}}\leq\delta\int_0^R\f{|m(y)|}{\tau_\e(y)}dy+\int_0^xe^{\int_0^{x'}\f{\nu+\beta_\e}{\tau_\e}}\int_0^R\f{\beta_\e(y)}{\tau_\e(y)}\kappa_\e(x',y)|m(y)|\,dydx'$$
and thus
\begin{eqnarray*}
|n(x)|&\leq&\delta\int_0^R\f{|m(y)|}{\tau_\e(y)}dy+\int_0^xe^{-\int_{x'}^x\f{\nu+\beta_\e}{\tau_\e}}\int_0^R\f{\beta_\e(y)}{\tau_\e(y)}\kappa_\e(x',y)|m(y)|\,dydx'\\
&\leq&\|m\|_E\,\f1\mu\Bigl[\delta R+\int_0^xe^{-\int_{x'}^x\f{\nu+\beta_\e}{\tau_\e}}\int_0^R\beta_\e(y)\kappa_\e(x',y)\,dydx'\Bigr]\\
&\leq&\|m\|_E\,\f1\mu\Bigl[\delta R+\Bigl\|\int_0^R\beta_\e(y)\kappa_\e(.,y)dy\Bigr\|_{L^\infty}\int_0^xe^{-\f\nu{\|\tau_\e\|_{L^\infty}}(x-x')}\,dx'\Bigr]\\
&\leq&\|m\|_E\,\underbrace{\f1\mu\Bigl[\delta R+\nu^{-1}\|\tau_\e\|_{L^\infty}\Bigl\|\int_0^R\beta_\e(y)\kappa_\e(.,y)dy\Bigr\|_{L^\infty}\Bigr]}_{:=k}.
\end{eqnarray*}
Because $\delta R<\mu$ by assumption, we can choose $\nu$ large so that $k<1$ and we obtain
$$\|n\|_E\leq k\|m\|_E.$$
Thus $T$ is a strict contraction and we have proved the existence of a solution to \eqref{eq:defopA}.

\paragraph{\it Second step: A is continuous.}
This relies on a general argument which in fact shows that the linear mapping $A$ is Lipschitz continuous. Indeed, arguing as above
$$|n(x)|e^{\int_0^x\f{\nu+\beta_\e}{\tau_\e}}\leq\delta\int_0^R\f{|n(y)|}{\tau_\e(y)}dy+\int_0^xe^{\int_0^{x'}\f{\nu+\beta_\e}{\tau_\e}}\int_0^R\f{\beta_\e(y)}{\tau_\e(y)}\kappa_\e(x',y)|n(y)|\,dydx'+\int_0^xe^{\int_0^{x'}\f{\nu+\beta_\e}{\tau_\e}}\f{|f(x')|}{\tau_\e(x')}\,dx',$$
and thus
$$|n(x)|\leq k\|n\|_E+\int_0^R\f{|f(x')|}{\tau_\e(x')}\,dx'\leq k\|n\|_E+\f R\mu\|f\|_E.$$
This indeed proves that
$$\|n\|_E\leq\f{R}{\mu(1-k)}\|f\|_E.$$

\paragraph{\it Third step: A is strongly positive.}
For $f\geq0,$ the operator $T$ of the first step maps $m\geq0$ to $n\geq0.$ Therefore the fixed point $n$ is nonnegative. In other words $n=A(f)\geq0.$ If additionally $f$ does not vanish, then $n$ does not vanish either. Therefore $n(0)=\delta\int_0^R\f{n(y)}{\tau_\e(y)}dy>0$ and thus
$$n(x)\geq n(0)+e^{-\int_0^x\f{\nu+\beta_\e}{\tau_\e}}\int_0^xe^{-\int_0^{x'}\f{\nu+\beta_\e}{\tau_\e}}\f{f(x')}{\tau_\e(x')}\,dx'>0.$$

\paragraph{\it Fourth step: A is compact.}
For $\|f\|_E\leq1,$ the third step proves that $n$ is bounded in $E$ and thus
$$\f{\p}{\p_x}n=-\f{\nu+\beta_\e}{\tau_\e}n+\int\f{\beta_\e(y)}{\tau_\e(y)}\kappa(x,y)n(y)dy+\f{f}{\tau_\e}$$
is also bounded in $E.$ Therefore by the Ascoli-Arzela theorem the family $n$ is relatively compact in $E.$

\

\paragraph{Adjoint equation.}
A function $\phi$ is a solution to the adjoint equation of \eqref{eq:regularized} if and only if $\tilde\phi(x):=\phi(R-x)$ satisfies
\beq\label{eq:phisym}\left \{ \begin{array}{l}
\displaystyle \tilde\tau_\e(x) \f{\p}{\p x} \tilde\phi(x) + ( \tilde\beta_\e(x) + \lb_\e) \tilde\phi(x) - 2 \tilde\beta_\e(x) \int_0^R\kappa_\e(y,R-x) \tilde\phi_\e(y)\,dy = \delta\tilde\phi_\e(R),\quad0<x<R,
\\
\\
\tilde\phi_\e(0)=0,
\end{array} \right.\eeq
where $\tilde\tau_\e(x)=\tau_\e(R-x)$ and $\tilde\beta_\e(x)=\beta_\e(R-x).$ Then the same method than for the direct equation give the result, namely the existence of $\lb$ and $\tilde\phi$ solution to \eqref{eq:phisym}.

\

Finally we have proved existence of $(\lb_\U,\U_\e)$ and $(\lb_\phi,\phi_\e)$ solution to the direct and adjoint equations of \eqref{eq:regularized}. It remains to prove that $\lb_\U=\lb_\phi$ but it is nothing but integrating the direct equation against the ajoint eigenvector, what gives $$\lb_\U\int\U_\e\phi_\e=\lb_\phi\int\U_\e\phi_\e.$$

\qed
\end{proof}

\

To have existence of solution for \eqref{eq:truncated}, it remains to do $\e\to0.$ For this we can prove uniform bounds in $L^\infty$ for $\U_\e$ and $\phi_\e$ because we are on the fixed compact $[0,R].$ Then we can extract subsequences which converge $L^\infty *-$weak toward $\U$ and $\phi,$ solutions to \eqref{eq:truncated} because $\tau_\e$ and $\beta_\e$ converge in $L^1$ toward $\tau$ and $\beta.$ Concerning $\kappa_\e,$ we have that for all $\varphi\in{\mathcal C}_c^\infty,\ \int\varphi(x)\kappa_\e(x,y)dx\longrightarrow\int\varphi(x)\kappa(x,y)dx\quad\forall y,$ and it is sufficient to pass to the limit in the equations.

\

\section{Maximum principle}

\begin{lemma}\label{lm:supersolution}
If there exists $A_0>0$ such that $\overline\phi\geq\phi$ on $[0,A_0]$ and $\overline\phi$ a supersolution of \eqref{eq:eigenproblem} on $[A_0,R]$ with $\overline\phi(R)\geq\phi(R),$ then $\overline\phi\geq\phi$ on $[0,R].$
\end{lemma}

\begin{proof}The proof is based on the same tools than to prove uniqueness (see above) or to establish GRE principles (see \cite{MMP1,MMP2} for instance).\\
We know that $\overline\phi\geq\phi$ on $[0,A_0]$ and that $\overline\phi$ is a supersolution to the equation satisfied by $\phi$ on $[A_0,R],\ i.e.$ there exists a function $f\geq\delta\phi(0)$ such that
$$-\tau(x)\f{\p}{\p x}\overline\phi(x)+(\lb+\beta(x))\overline\phi(x)=2\beta(x)\int_0^x\kappa(y,x)\overline\phi(y)\,dy+f(x),\quad\forall x\in[A_0,R].$$
So we have for all $x\in[A_0,R]$
$$-\tau(x)\f{\p}{\p x}(\phi(x)-\overline\phi(x))+(\lb+\beta(x))(\phi(x)-\overline\phi(x))=2\beta(x)\int_0^x\kappa(y,x)(\phi(y)-\overline\phi(y))\,dy-f(x).$$
Then, multiplying by $\1_{\phi\geq\overline\phi},$ we obtain (see \cite{BP} for a justification)
$$-\tau(x)\f{\p}{\p x}(\phi-\overline\phi)_+(x)+(\lb+\beta(x))(\phi-\overline\phi)_+(x)\leq2\beta(x)\int_0^x\kappa(y,x)(\phi-\overline\phi)_+(y)\,dy-f(x)\1_{\phi\geq\overline\phi}(x),$$
and this inequality is satisfied on $[0,R]$ since $(\phi-\overline\phi)_+\equiv0$ on $[0,A_0].$\\
If we test against $\U$ we have, using the fact that $\phi(R)=0<\overline\phi(R),$
\begin{eqnarray*}
\lefteqn{\int_0^R(\phi-\overline\phi)_+(x)\f{\p}{\p x}(\tau(x)\U(x))\,dx+\int_0^R(\lb+\beta(x))(\phi-\overline\phi)_+(x)\U(x)\,dx}\\
&\leq&2\int_0^R(\phi-\overline\phi)_+(y)\int_y^R\beta(x)\kappa(y,x)\U(x)\,dxdy-\int_0^Rf(x)\1_{\phi\geq\overline\phi}(x)\U(x)\,dx.
\end{eqnarray*}
But if we test the equation \eqref{eq:eigenproblem} satisfied by $\U$ against $(\phi-\overline\phi)_+,$ we find
\begin{eqnarray*}
\lefteqn{\int_0^R(\phi-\overline\phi)_+(x)\f{\p}{\p x}(\tau(x)\U(x))\,dx+\int_0^R(\lb+\beta(x))(\phi-\overline\phi)_+(x)\U(x)\,dx}\hspace{2cm}\\
&=&2\int_0^R(\phi-\overline\phi)_+(y)\int_y^R\beta(x)\kappa(y,x)\U(x)\,dxdy,
\end{eqnarray*}
and finally, substracting,
$$0\leq-\int_0^Rf(x)\1_{\phi\geq\overline\phi}(x)\U(x)\,dx,$$
so
$$\delta\phi(0)\int_0^R\1_{\phi\geq\overline\phi}(x)\U(x)\,dx\leq0$$
and this can hold only if $\1_{\phi\geq\overline\phi}\equiv0$ or $\phi(0)=0.$ But we deal with the truncated problem with $\tau(x)\geq\eta>0,$ so $\f1\tau\in L^1_0$ and $\phi(0)>0$ thanks to the lemma \ref{lm:positivity}. Thus $\1_{\phi\geq\overline\phi}\equiv0$ and the lemma \ref{lm:supersolution} is proved.
\qed
\end{proof}

\bigskip

\bibliographystyle{plain}
\bibliography{Prion}

\end{document}